\documentclass[10pt,one side,emlines]{amsart}
\usepackage{amssymb,latexsym,xy,eucal,mathrsfs}
\usepackage{tikz}
\usetikzlibrary{arrows.meta}
\usepackage{graphicx, graphics}
\usepackage{caption}
\usepackage{wrapfig}
\newtheorem{lem}{Lemma}[section]
\newtheorem{lemma}[lem]{Lemma}
\newtheorem{theorem}[lem]{Theorem}
\newtheorem*{theorem*}{Main Theorem}
\newtheorem{thm}[lem]{Theorem}

\newtheorem{corollary}[lem]{Corollary}

\theoremstyle{definition}

\newtheorem*{cthm*}{Conjecture}
\newtheorem{rem}[lem]{Remark}

\newtheorem{defn}[lem]{Definition}

\begin{document}
	\baselineskip 15truept
	
	
	\title{Complemented zero-divisor graph of posets}
	\maketitle
	\markboth{Anagha Khiste, Ganesh Tarte and Vinayak Joshi}{Complemented zero-divisor graph of posets}\begin{center}\begin{large} $^\text{a}$Anagha Khiste, $^\text{b}$Ganesh Tarte and $^\text{c}$Vinayak Joshi \end{large}\begin{small}\vskip.1in$^\text{a}$\emph{Department of Applied Science and Humanities, Indian Institute of Information Technology, Pune - 410507. }\\$^\text{b}$\emph{Department of Applied Sciences and Humanities, Pimpri Chinchwad College of Engineering, Pune - 411044.}\\$^\text{c}$\emph{Department of Mathematics, Savitribai Phule Pune University, Pune - 411007, Maharashtra, India.}\\\emph{E-mail: avanikhiste@gmail.com, ganesh.tarte@pccoepune.org, vvjoshi@unipune.ac.in }\end{small}\end{center}

\begin{abstract}
	In this paper, we derive a set of equivalent conditions for the zero-divisor graph $\Gamma(Q)$ of a poset $Q$ with $0$ to be complemented, characterizing it in terms of quasi-complemented posets. Furthermore, we prove that the notions of a complemented zero-divisor graph and a uniquely complemented zero-divisor graph coincide for any poset $Q$ with $0$. In addition, we provide both algebraic and topological characterizations for $\Gamma(Q)$ to be a complemented graph. In the final section, we apply these characterizations to the zero-divisor graphs of a reduced (multiplicative) semigroup $S$ with $0$ and the comaximal (ideal) graph  of an Artinian ring $R$, and the nonzero component union graph $\mathbb{UG}(\mathbb{V})$ of a finite-dimensional vector space $\mathbb{V}$ over a field $\mathbb{F}$.
\end{abstract}

\maketitle \noindent{ \small \textbf{Keywords}: Comaximal (ideal) graph,  complemented graph,   minimal prime ideal, nonzero component union graph, quasi-complemented poset and zero-divisor
	graph· 
} \maketitle

\section{Introduction}
The concept of the zero-divisor graph of a commutative ring with identity was introduced by Beck \cite{B} as a tool to explore the interplay between ring-theoretic and graph-theoretic properties. In recent years, zero-divisor graphs have been extensively studied across various algebraic structures, as seen in the works of Anderson and Naseer \cite{AN}, Anderson and Livingston \cite{AL}, and DeMeyer, McKenzie, and Schneider \cite{DMS}. These investigations have also extended into ordered structures, with notable contributions from Hala\v{s} and Jukl \cite{HJ}, Joshi et al. \cite{VK, JA 1}, Nimbhorkar et al. \cite{NWD}, and Lu and Wu \cite{LW}.

In \cite[Theorem 3.5]{ALS}, Anderson et al. proved that for a reduced ring $R$, the zero-divisor graph is uniquely complemented if and only if it is complemented, if and only if the total ring of quotients $T(R)$ is von Neumann regular. It is well known that the zero-divisor graph of a Boolean ring is isomorphic to the zero-divisor graph of the Boolean lattice derived from it. LaGrange \cite{L} further investigated the complementedness of zero-divisor graphs in Boolean rings, and this study was later extended to Boolean posets by Joshi and Khiste \cite{JA 2}.

In 2012, Visweswaran \cite{SV} examined the complementedness of the complement of the zero-divisor graph of a commutative ring with identity. More recently, Bennis et al. \cite{BAL} explored complemented and uniquely complemented properties for extended zero-divisor graphs of commutative rings, while Bender et al. \cite{CPRL} analyzed similar properties for zero-divisor graphs associated with finite commutative semigroups, particularly in relation to the clique number.

In this paper, we investigate a special class of posets, denoted by $\mathbb{P}^{\ell}_{MFP}$, consisting of posets with the least element $0$ in which every maximal filter is prime, and every maximal $\ell$-filter is maximal among all filters. It is known that a lattice $Q$ belongs to $\mathbb{P}^{\ell}_{MFP}$ if and only if it is a $0$-distributive lattice. This class has been well studied by Mundlik et al. \cite{MJH}.

The notion of quasi-complemented lattices was originally introduced by Grillet and Varlet \cite{GV}, and further investigated by Cornish \cite{C}. Jayaram \cite{Jr} studied quasi-complemented semilattices. More recently, Mundlik et al. \cite{MJH} extended these ideas to define and explore quasi-complemented and weakly quasi-complemented posets. We extend these developments by examining posets $Q \in \mathbb{P}^{\ell}_{MFP}$ that satisfy the annihilator condition. 

Specifically, we characterize the structure of the zero-divisor graph $\Gamma(Q)$ of a poset $Q$ with $0$ and present a set of equivalent conditions under which it becomes a complemented graph, using the framework of quasi-complemented posets. We prove that for any poset $Q$ with $0$, the notions of a complemented zero-divisor graph and a uniquely complemented zero-divisor graph are equivalent. In addition, we provide both algebraic and topological characterizations for $\Gamma(Q)$ to be complemented. We prove that $\Gamma(Q)$ is complemented (respectively, uniquely complemented) if and only if $Q$ is  quasi-complemented.

Finally, we apply our results to various structures, including the zero-divisor graphs of reduced semigroups with $0$, the comaximal ideal graph $\mathbb{CIG}(R)$ and the comaximal graph $\mathbb{CG}(R)$ of an Artinian ring $R$, and the nonzero component union graph $\mathbb{UG}(\mathbb{V})$ of a finite-dimensional vector space $\mathbb{V}$ over a field $\mathbb{F}$.

\section{Preliminary concepts and definitions}

Let $Q$ be a poset and $A \subseteq Q$. The set $A^u = \{ x \in Q~|~ a \leq x$ for every $a \in A\}$ is called the \textit{upper cone} of $A$. Dually, we have the concept of the \textit{lower cone} $A^{ \ell }$ of $A$. $A^{u \ell }$ shall mean $\{A^u\}^{ \ell }$ and $A^{ \ell u}$ shall mean $\{A^{ \ell }\}^u$. We use $a^u$ instead of $\{a\}^u$ and dually. We note that $A \subseteq A^{u \ell }$ and $A \subseteq A^{ \ell u}$. If $ A \subseteq B$ then $A^{ \ell } \supseteq B^{ \ell }$ and $A^u \supseteq B^u$. Moreover, $ A^{\ell u \ell } = A^{ \ell }$, $A^{ u \ell u} = A^u$, $\{a^u\}^{ \ell }=a^{ \ell }$, and $\{a^\ell\}^u=a^u$.

Given a poset $Q$ and a non-empty subset $I$ of $Q$, it is said to be a \textit{semi-ideal} if $y\in I$ and $x\leq y$ imply $x\in I$.
A non-empty subset $I$  of $Q$ is said an \textit{ideal} if $a, b\in I$ yields $\{a, b\}^{u\ell} \subseteq I$ (\textit{see} Hala\v s \cite{H}). The set of all ideals of $Q$ is denoted by $Id(Q)$. Dually, we have the concept of a \textit{semi-filter} and \textit{filter} respectively.
An ideal $I$ is said to be a \textit{principal ideal generated} by $a \in Q$ if $I = (a] = \{x \in Q ~|~ x \leq a\}$. Dually, we have the \textit{principal filter} $F$ generated by $a \in Q$ if $F =[a)=\{x \in Q ~|~ a \leq x\}$. An ideal $I$ is called a \textit{$u$-ideal} if, for all $x, y \in I$, $\{x, y\}^{u} \cap I \neq \emptyset$. Note that every principal ideal is a $u$-ideal but not conversely. Dually, we have the concepts of an \textit{$\ell$-filter}.

For a non-empty subset $A \subseteq Q$, the \textit{annihilator} of $A$ is denoted by $A^{\perp }= \{y \in Q | \{x, y\}^\ell = \{0\}, \forall x \in A\}$. In particular, if $A = \{x\}$, then the annihilator of $x$ is $x^{\perp}= \{y \in Q | \{x, y\}^\ell = \{0\}\}$ (\textit{see} Joshi and Waphare \cite{JW}). A proper ideal $I$ of $Q$ is called \textit{prime} if for $x, y \in Q$, $\{x, y\}^{\ell} \subseteq I$ implies $x \in I$ or $y \in I$. Dually, we have the concept of a \textit{prime filter}. The set of all prime ideals of a poset $Q$ is denoted by $Spec(Q)$. Minimal elements of the poset of all prime ideals (prime $u$-ideals) of $Q$ will be called \textit{minimal prime ideals (minimal prime $u$-ideals)} of $Q$. The set of all minimal prime ideals of a poset $Q$ with $0$ is denoted by $Min(Q)$. A proper ideal ($u$-ideal) $I$ is a \textit{maximal ideal ($u$-ideal)} if there is no proper ideal ($u$-ideal) $J$ such that $I\subsetneqq J\subsetneqq Q$. Dually, we have the concept of a \textit{maximal filter (maximal $\ell$-filter)}.

For $a\in Q$, we write $V (a) = \{P \in Spec(Q) ~|~ a \in P\}$ and $D(a) = \{P \in Spec(Q) ~|~ a\notin P\} = Spec(Q) \setminus V (a)$. We set $V'(a)=V(a)\cap Min(Q)$ and $D'(a)=D(a)\cap Min(Q)$. Then the sets $V(I)= \bigcap_{a\in I} V(a)$, where $I$ is an ideal of $Q$, satisfy the axioms for the closed sets of a topology on $Spec(Q)$, called the \textit{Zariski topology}. Given a subset $\mu$ of a topological space $Spec(Q)$,  the \textit{interior} of $\mu$ is denoted by $int(\mu)$ and defined as the union of all open sets contained in $\mu$ and the \textit{closure} of $\mu$ is denoted by $\overline{\mu}$ and defined as the intersection of all closed sets containing $\mu$, equivalently the closure of $\mu$ is the set of all points $x\in Spec(Q)$ such that every neighborhood of $x$ intersects $\mu$. We also consider $Min(Q)$ as a subspace of $Spec(Q)$.

We denote by $\mathbb{P}_{MFP}$ the class of posets having the least element $0$ in which every maximal filter is prime, and by $\mathbb{P}^{\ell}_{MFP}$ the subclass of $\mathbb{P}_{MFP}$ consisting of all posets having the least element $0$ and with the property that every maximal $\ell$-filter (i.e., maximal among all $\ell$-filters) is a maximal filter (i.e., maximal among all filters). Hence every maximal $\ell$-filter $F$ of a poset $Q$ in $\mathbb{P}^{\ell}_{MFP}$ is also prime. Dually, we have the class $\mathbb{P}^{u}_{MIP}$ (\textit{see} Mundlik et al. \cite{MJH}).

The concept of a quasi-complemented lattice was introduced and studied by Grillet and Varlet \cite{GV}, which was extended by Mundlik et al. \cite{MJH} to posets.

A poset $Q$ with $0$ is called \textit{quasi-complemented} if for any $x\in Q$, there exists $y$ distinct from $x$ such that $\{x,y\}^{\ell} = \{0\}$ and $ \bigg((x] \vee (y]\bigg)^{\perp } = \{0\}$. A poset $Q$ is said to be \textit{weakly quasi-complemented} if for any $x \in Q$, there exist $y_1, y_2,\cdots, y_n$ distinct from $x$ such that $x^{\perp \perp} = \bigcap_{i=1}^{n}y_{i}^{\perp }$.

One can see that a modular, non-distributive lattice $M_3$ is quasi-complemented. However, $M_3 \notin \mathbb{P}^{\ell}_{MFP}$. Now, we provide an example of a lattice $L\in \mathbb{P}^{\ell}_{MFP}$ but not quasi-complemented. Consider the set $L=\{X\subseteq \mathbb{N}~:~|X|< \infty\}$. Then $L$ is a distributive lattice and hence $L\in \mathbb{P}^{\ell}_{MFP}$. However, $L$ is not quasi-complemented. Thus, the class of posets $\mathbb{P}^{\ell}_{MFP}$ and quasi-complemented are not contained in one another. Further, the class of poset $\mathbb{P}^{\ell}_{MFP}$ is a subclass of $0$-distributive poset, and these two classes coincide if $P$ happens to be a lattice; \textit{see} Theorem \ref{2.3}.

A poset $Q$ with the least element $0$ is said to be \textit{$0$-distributive} if for $x, y, z \in Q$, $\{x, y\}^{\ell} = \{x,z\}^{\ell} = \{0\}$ imply $\{x, \{y,z\}^u\}^{\ell} = \{0\}$ (\textit{see} Joshi and Waphare \cite{JW}). Let $Q$ be a poset with $0$. An element $x^* \in Q$ is said to be the \textit{pseudocomplement} of $x \in Q$, if $\{x, x^*\}^{\ell} = \{0\}$ and for $y \in Q$, $\{x, y\}^{\ell} =\{0\}$ implies $y\leq x^*$. A poset $Q$ with $0$ is called \textit{pseudocomplemented} if each element of $Q$ has the pseudocomplement (\textit{see} Venkatanarasimhan \cite{V} and Hala\v s \cite{H}).

Let $Q$ be a poset with $0$. An element $x \in Q$ is called a \textit{zero-divisor} if $\{x,y\}^\ell=\{0\}$ for some $y \in Q\setminus \{0\}$. Let $Z(Q)$ denote the set of all zero-divisors of $Q$. For a poset $Q$ with $0$, we associate an undirected graph called the \textit{zero-divisor graph of $Q$}, denoted by $\Gamma(Q)$, whose vertex set $V(\Gamma(Q))$ consists of the non-zero zero-divisors of $Q$ and two distinct vertices $x,y$ are adjacent if and only if $\{x, y\}^ {\ell} = \{0\}$; (\textit{see} Lu and Wu \cite{LW} and Joshi \cite{J}).

For distinct vertices $a$ and $b$ in a graph $G$, we say that $b$ is a \textit{complement} of $a$, written $a\bot b$ if $a$ and $b$ are adjacent and there is no vertex $c$ in $G$ which is adjacent to both $a$ and $b$. A graph $G$ is called \textit{complemented} if for each vertex $a$ of $G$, there is a vertex $b$ of $G$ such that $a\bot b$, and that $G$ is an \textit{uniquely complemented} if $G$ is complemented and whenever $a\bot b$ and $a\bot c$, then $b$ and $c$ are adjacent to exactly the same vertices; {\it see} LaGrange\cite{L}.

 Let $G=(V,E)$ be a simple graph. The \textit{complement of $G$} denoted by $G^c$ is defined as a graph with $V(G^c)= V(G)=V$ and two distinct elements $x, y\in V$ are joined by an edge in $G^c$ if and only if there is no edge of $G$ joining $x$ and $y$; see West \cite{W}.

For undefined concepts in lattices and graphs, {\it see} \cite{G, W} respectively.

We need the following results to prove  Main Theorem.

\begin{theorem}[Mundlik et al. \cite {MJH}]\label{2.3'}
	Let $Q$ be a poset in $\mathbb{P}^{\ell}_{MFP}$. Then the following statements are true.
	\begin{enumerate}
		\item[i)] $Q$ is a $0$-distributive poset.
		\item[ii)] $Spec(Q)\neq \emptyset$.
		\item[iii)] A prime ideal $P$ is minimal if and only if $P$ contains precisely one of $(x]$ or $x^{\perp }$ for any $x \in Q$.
		\item[iv)] $Q$ is quasi-complemented if and only if for any $x \in Q$ there exists $y \in Q$ such that $x^{\perp}= y^{\perp \perp}$.
	\end{enumerate}
\end{theorem}

\begin{theorem}[Pawar and Thakare \cite{PT 1}, Theorem 5]\label{2.3}
	If $Q$ is a meet-semilattice, then $Q\in \mathbb{P}^{\ell}_{MFP}$ if and only if $Q$ is $0$-distributive.
\end{theorem}

\begin{theorem}[Joshi and Waphare \cite{JW}, Theorem 2.10]\label{2.4'}
	Every pseudocomplemented poset is $0$-distributive. Further, a poset $Q$ is $0$-distributive if and only if the lattice $Id(Q)$ is pseudocomplemented.
\end{theorem}

\begin{lemma}[Joshi and Mundlik \cite{VM}, Lemma 3.2]\label{3.8}
	Every pseudocomplemented poset is quasi-complemented.
\end{lemma}

With the groundwork laid out, let's proceed to the first result.

\begin{thm}\label{1.6} 
	Let $Q$ be a poset in $\mathbb{P}^{\ell}_{MFP}$. Then the intersection of all minimal prime ideals in $Q$, denoted $\bigcap Min(Q)$, is $\{0\}$, and hence the intersection of all prime ideals of $Q$ is also $\{0\}$.
\end{thm}
\begin{proof} 
	Let $Q \in \mathbb{P}^{\ell}_{MFP}$. Suppose, to the contrary, that there exists a non-zero element $x \in \bigcap Min(Q)$. This implies $(0] \cap [x)= \emptyset $. Then by Zorn's lemma, there exists a maximal $\ell$-filter $F$ containing $x$ such that $F \cap (0]=\emptyset $. Since $Q\in \mathbb{P}^{\ell}_{MFP}$, $F$ is a prime filter with $x \in F$. As $F$ is a maximal $\ell$-filter, which is also a prime filter, $Q\setminus F$ is a prime ideal. Now, we prove that $Q \setminus F$ is a minimal prime ideal. Suppose, on the contrary, there exists a prime ideal $P$ such that $(0] \subsetneqq P \subsetneqq Q \setminus F$. Then it is evident that $F \subsetneqq Q \setminus P$ and $Q \setminus P$ is an $\ell$-filter, which contradicts the maximality of $F$. Hence, we have $Q \setminus F$ as a minimal prime ideal, implying $x \notin Q \setminus F$. This contradicts the assumption that $x \in \bigcap Min(Q)$. Therefore, $\bigcap Min(Q)=\{0\}$. As $\bigcap Spec(Q)\subseteq \bigcap Min(Q)$, we conclude that $\bigcap Spec(Q)=\{0\}$.
\end{proof}

\begin{lemma}\label{2.9} 
	Let $Q$ be a poset in $\mathbb{P}^{\ell}_{MFP}$. Then $int V(x)=(Spec(Q)\setminus \overline{D(x)})$ for any $x\in Q$.
\end{lemma}
\begin{proof} 
	Let $x\in Q$ be any element. First, we prove that $int V(x)\subseteq (Spec(Q)\setminus \overline{D(x)})$. Let $P\in int V(x)$. By the definition of the interior operator, $P\in D(a)$ for some open set $D(a)\subseteq V(x)$. Hence $D(a)\cap D(x) = \emptyset$. This implies the existence of a neighborhood $D(a)$ of $P$ such that $D(a)\cap D(x)=\emptyset $. Therefore, by the definition of the closure operator $P\notin \overline{D(x)}$, i.e., $P\in Spec(Q)\setminus \overline{D(x)}$. Hence $int V(x) \subseteq (Spec(Q)\setminus \overline{D(x)})$. Now, let $P\in Spec(Q)\setminus \overline{D(x)}$. Thus, there exists a neighborhood $D(a)$ of $P$ such that $D(a)\cap D(x) = \emptyset $. We claim that $\{a,x\}^{\ell}=\{0\}$. Suppose $\{a,x\}^{\ell}\neq \{0\}$. Then there exists a non-zero element $t$ such that $t\in \{a, x\}^{\ell}$. Since $Q\in \mathbb{P}^{\ell}_{MFP}$, the intersection of all prime ideals is $\{0\}$. Hence, there exists a prime ideal $P_1$ such that $t\not \in P_1$. But then $\{a, x\}^{\ell}\not \subseteq P_1$, implies that $P_1\in D(a)\cap D(x)$.  This contradicts the fact that $D(a)\cap D(x)=\emptyset $. Thus $\{a, x\}^{\ell}=\{0\}$, implying $D(a)\subseteq V(x)$. This implies $P\in int V(x)$. Hence, $(Spec(Q)\setminus \overline{D(x)})\subseteq int V(x)$. Therefore, $int V(x)=(Spec(Q)\setminus \overline{D(x)}).$ 
\end{proof}

\begin{lemma}[Mundlik et al. \cite{MJH}, Lemma 2.7]\label{2.9'} 
	Let $Q$ be a poset with $0$ and $X$ be any non-empty subset of $Spec(Q)$. Then $\overline {X} = V(\bigcap _{P\in X} P)$.
\end{lemma}

\begin{thm}\label{1.5}
	Let $Q$ be a poset in $\mathbb{P}^{\ell}_{MFP}$. Then the following are equivalent. \begin{enumerate}\item $Q$ is a quasi-complemented poset. \item For every $x \in Q$, there exists $y \in x^{\perp}$ such that $x^{\perp} \cap y^{\perp}=\{0\}$. \item For every $x \in Q$, there exists $y \in Q$ such that $\overline{int(V(x))}=\overline{D(y)}$. 
\end{enumerate}\end{thm}

\begin{proof} $(1)\Rightarrow (2)$: Let $Q$ be a quasi-complemented poset and let $x \in Q$ be any element.
	Then by Lemma \ref{2.3'}, there exist $y \in Q$ such that $x^{\perp}= y^{\perp \perp}$. It is easy to observe that $y\in y^{\perp \perp}=x^{\perp}$ and $x^{\perp} \cap y^{\perp}=y^{\perp \perp} \cap y^{\perp}=\{0\}$.   
	\vskip 5truept 
	$(2)\Rightarrow (3)$: Let $x\in Q$ be any element. Therefore by the assumption, there exists $y\in x^{\perp}$ such that $x^{\perp} \cap y^{\perp}=\{0\}$. First, we claim that $V'(x)\cap V'(y)=\emptyset$. For
	this assumes that there exists a minimal prime ideal
	$P_{0}$ such that $P_{0}\in V'(x) \cap V'(y)$. By Theorem \ref{2.3'}, we have $x^{\perp} \nsubseteq  P_0$  and $y^{\perp} \nsubseteq P_0$. Hence there exist $z_1 \in x^{\perp}$ and $z_2 \in y^{\perp}$  such that $z_1, z_2 \notin P_0$. This together with $\{z_1,z_2\}^{\ell} \subseteq x^{\perp} \cap y^{\perp} =\{0\}\subseteq P_0$, which is a contradiction to the fact that $z_1, z_2 \notin P_0$.
	
		{Thus 	$V'(x)\cap V'(y)=\emptyset$. ---------------------------------------------------\textbf{$(\star)$}}
	\vskip 5truept

	Now, we claim that $\overline{intV(x)}=\overline{D(y)}$. For this, let $P\in \overline{int(V(x))}.$  By Lemma \ref{2.9}, $P\in \overline{(Spec(Q)\setminus \overline{D(x)})}.$ Let $D(a)$ be any neighborhood of $P$. Using the definition of the closure operator, $ D(a)\cap (Spec(Q)\setminus
	\overline{D(x)})\neq \emptyset.$ Thus there exists a prime ideal $P_{1}$ of
	$Q$ such that $P_{1}\in D(a)\cap (Spec(Q)\setminus
	\overline{D(x)})$. Since $P_{1}\not \in \overline{D(x)}$, we have
	a neighborhood $D(b)$ of $P_{1}$ such that $D(b)\cap
	D(x)=\emptyset$. Hence $\{b,x\}^{\ell}=\{0\}$. As $a, b\not \in
	P_{1}$, we have $\{a,b\}^{\ell} \neq \{0\}$. Hence there exists a non-zero element $z \in \{a,b\}^{\ell}$. By Theorem \ref{1.6}, there exists a minimal prime ideal $P_{2}$ such that $z \not \in P_{2}$. Thus $a,b\not\in P_2$. 
	Since $\{b,x\}^{\ell}=\{0\} $ implies that $x\in P_2,$ \textit{i.e.,} $P_{2}\in V'(x).$ Further $V'(x)\cap V'(y)=\emptyset$, we have $P_2 \not\in V'(y),$ \textit{i.e.,} $y\not \in P_{2}$. Therefore $P_{2}\in D(y)\cap D(a)$. Thus $D(y)\cap D(a)\neq \emptyset$ and hence $P\in	\overline{D(y)}$.  This proves $\overline{intV(x)}\subseteq \overline{D(y)}$.
	\vskip 5truept 
	Now, we suppose that there exists a prime ideal $P\in \overline{D(y)}$ such that $P\not \in \overline{intV(x)}=\overline{(Spec(Q)\setminus
	\overline{D(x)})}$.	By Lemma \ref{2.9'},  $P\notin V(\bigcap_{R \in (Spec(Q)\setminus \overline{D(x)})} R),$ \textit{i.e.,} $\bigcap_{R \in (Spec(Q)\setminus \overline{D(x)})} R\not\subseteq P$. Thus there exists an element $a \in \bigcap_{R\in (Spec(Q)\setminus
	\overline{D(x)})}R $ such that $a \not \in P$. Since $P\in \overline{D(y)}\cap D(a)$, we have $D(a)\cap D(y)\neq \emptyset$.
	Thus, there is a prime ideal $R_1\in D(a)\cap D(y)$. As $a\in
	\bigcap_{R\in (Spec(Q)\setminus \overline{D(x)})} R $ and $R_1\in
	D(a)$ give $R_1\in \overline{D(x)}$. This, together with $R_1\in
	D(y)$, we have $D(x)\cap D(y)\neq \emptyset$. This gives a
	contradiction to the fact that $\{x,y\}^{\ell} = \{0\}$, \textit{i.e.}, $y\in
	x^{\perp}$. Thus $\overline{D(y)}\subseteq \overline{intV(x)}.$ Hence $\overline{intV(x)} = \overline{D(y)}$.
	\vskip 5truept 
	$(3)\Rightarrow (1)$: Let $x \in Q$ be any element. Then by the assumption, there exists $y \in Q$ such that $\overline{int V(x)}= \overline{D(y)}$. Now, to prove $Q$ is a quasi-complemented poset, we claim that $y \in x^{\perp}$ such that $((x]\vee (y])^{\perp}=\{0\}$. Suppose that $y \notin x^{\perp}$.  This gives $z \in \{x,y\}^{\ell}$ such that $z \neq 0$. By Theorem \ref{1.6}, there is a prime ideal $P$ such that $z \notin P$. Thus, $P \in D(x) \cap D(y)$.
	Since $P \in D(y)$ and $D(y) \subseteq \overline {D(y)}=\overline {int V(x)}$, we have $D(x) \cap int V(x) \neq \emptyset$,  a contradiction to $D(x) \cap V(x) = \emptyset$. Hence $y \in x^{\perp}$.

	Now, since $Q \in \mathbb{P}^{\ell}_{MFP}$, we have $Q$ is a $0$-distributive poset by Theorem \ref{2.3'}. Applying Theorem \ref{2.4'}, we have $Id(Q)$ as a pseudocomplemented lattice. It is known that in a pseudocomplemented lattice, $(I \vee J)^{\perp } = I^{\perp } \cap  J^{\perp }$ for any $I, J \in Id(Q)$. This implies that $((x] \vee (y])^{\perp}=x^{\perp} \cap y^{\perp}$. So, let $z \in  ((x] \vee (y])^{\perp}=x^{\perp} \cap y^{\perp}$ such that $z \neq 0$. Clearly, $D(x) \cap D(z)=\emptyset = D(y) \cap D(z)$. Since $z \neq 0$, there exists a prime ideal $P$ such that $z \notin P$. This together with  $D(x) \cap D(z)=\emptyset$ gives $P \notin \overline{D(x)}$ and hence by Lemma \ref{2.9},  $P \in int V(x) \subseteq \overline{int V(x)}=\overline{D(y)}.$ This together with $P\in D(z),$ we have $D(z) \cap D(y) \neq \emptyset,$  a contradiction to $D(z) \cap D(y) = \emptyset$. Thus $((x] \vee (y])^{\perp} = \{0\}$. Hence, $Q$ is a quasi-complemented poset.
\end{proof}

\begin{corollary}\label{1.8}
	Let $Q$ be a poset in $\mathbb{P}^{\ell}_{MFP}$. Then for every $x \in Q$, there exists $y \in x^{\perp}$ such that $x^{\perp} \cap y^{\perp}=\{0\}$ if and only if for every $x \in Q$, there exists $y \in Q$ such that $V'(y)=V'(x^{\perp})$.
\end{corollary}

\begin{proof} The only if implication follows by $(\star)$ of Theorem \ref{1.5} and Theorem \ref{2.3'}. For the converse part, let $x\in Q$. Then by the assumption, there exists $y\in Q$ such that $V'(y)= V'(x^{\perp})$. To prove that $y\in x^{\perp}$, consider a non-zero element $z\in \{x, y\}^{\ell}$. Now, using the fact that $\bigcap Min(Q)=\{0\}$, there exists $M\in Min(Q)$ such that $z\notin M$. This gives that $x,y\notin M$ and hence $M\in V'(x^{\perp})=V'(y)$, which is not possible. Thus, $y\in x^{\perp}$. On similar lines, we can prove that $x^{\perp}\cap y^{\perp}=\{0\}$.
\end{proof}

It should be noted that the complemented zero-divisor graphs of rings need not be uniquely complemented. However, in the case of posets, we have a pleasant situation. 

\begin{thm}\label{1.11}
	Let $Q$ be a poset with $0$. Then $\Gamma(Q)$ is complemented if and only if $\Gamma(Q)$ is uniquely complemented.
\end{thm}
\begin{proof} Let $\Gamma(Q)$ be a complemented graph and $a, b, c \in V(\Gamma(Q))$ such that $a\bot b$ and $a\bot c$. Assume that $x \in V(\Gamma(Q))$ be adjacent to $b$ but not adjacent to $c$. Therefore $\{x,c\}^{\ell}\neq \{0\}$. Thus, there exists $y \neq 0$ such that $y \in \{x,c\}^{\ell}$. Since  $ \{x,b\}^{\ell}=\{0\}=\{a,c\}^{\ell},$ we have  $y \in V(\Gamma(Q))\setminus_{\{a,b\}}$ such that $ \{y,b\}^{\ell}=\{0\}=\{y,a\}^{\ell}$ and hence $y$ is adjacent to both $a$ and $b$, contradicts to $ a \bot b$. Hence, $x$ is adjacent to $c$. Similarly, we can prove that if $x$ is adjacent to $c$, then $x$ is adjacent to $b$. Thus, $b$ and $c$ are adjacent to exactly the same vertices. Therefore,  $\Gamma(Q)$ is uniquely complemented. Converse follows by the definition of a uniquely complemented graph.
\end{proof}

\begin{defn} A poset $Q$ with $0$ is said to satisfy the  \textit{annihilator condition $(a.c.)$}, if for every $ x,y \in Q$, there exists $z \in Q$  such that $x^{\perp} \cap y^{\perp}=z^{\perp}$.
\end{defn}	

\begin{thm}\label{1.11'} Let $Q$ be a poset in $\mathbb{P}^{\ell}_{MFP}$ such that $Q$ satisfies $a.c.$ Then $Q$ is a quasi-complemented poset if and only if $\Gamma(Q)$ is complemented.  
\end{thm}               

\begin{proof} Let $Q$ be a quasi-complemented poset. It is easy to prove that  $\Gamma(Q)$ is complemented using Theorem \ref{1.5} and the definition of the complement of a vertex in a graph. Conversely, let $\Gamma(Q)$ be a complemented graph. To prove that, $Q$ is a quasi-complemented, by Theorem \ref{1.5}, we claim that for every $x\in Q$, there exists $y \in x^{\perp}$ such that $x^{\perp}\cap y^{\perp}=\{0\}$. Let $x\in Q$. If $x\in V(\Gamma(Q))$, then by the definition of a complemented graph, there exists $y\in V(\Gamma(Q))$ such that $x\bot y$, {\it{i.e.,}} $x$ and $y$ are adjacent in $\Gamma(Q)$ such that there is no vertex $z (\neq 0)$, which is adjacent to both $x$ and $y$. Hence $y \in x^{\perp}$ with $x^{\perp}\cap y^{\perp}=\{0\}$. We assume that $x\notin V(\Gamma(Q))$. This implies that either $x^{\perp}=\{0\}$ or $x=0$. \newline If $x^{\perp}=\{0\}$, then we have $y=0\in x^{\perp}$ such that $x^{\perp}\cap y^{\perp}=\{0\}$.\newline Let $x=0$. If $V(\Gamma(Q)) = \emptyset$, then for any non-zero element $y\in Q$, we have $y^{\perp }=\{0\}$ and hence $x \in y^\perp$ and $x^\perp \cap y^\perp=\{0\}$. Now, we consider the case $V(\Gamma(Q))\neq \emptyset $. Since $\Gamma(Q)$ is a complemented graph, there exist $a,b\in V(\Gamma(Q))$ such that $a \bot b$, that is, $a^{\perp}\cap b^{\perp}=\{0\}$. Since $Q$ satisfies $a.c.$, there exists $y\in Q$ such that $a^{\perp} \cap b^{\perp}=y^{\perp}$. Clearly $y\neq 0$. Thus,  $x \in \{0\} = y^{\perp}$ and $x^{\perp}\cap y^{\perp}=\{0\}$. 
\end{proof}          

If $Q$ is a lattice and $Q\in \mathbb{P}^{\ell}_{MFP}$, by Theorem \ref{2.3}, $Q$ is $0$-distributive. Further, in a  $0$-distributive lattice $L$, for any $a,b\in L$, $(a\vee b)^{\perp}=a^{\perp}\cap b^{\perp}$. This shows that a $0$-distributive lattice always satisfies $a.c.$ So  Theorem \ref{1.11'} takes a following form.

\begin{corollary}\label{1.12'}
	Let $Q$ be a lattice such that $Q\in \mathbb{P}^{\ell}_{MFP}$ (equivalently, $Q$ is a $0$-distributive lattice). Then $Q$ is  quasi-complemented if and only if $\Gamma(Q)$ is complemented.       
\end{corollary}                                                          
A topological space $X$ is  {\it compact} if for every
family of open sets $\{A_{\alpha}\}_{\alpha \in \Lambda} $ with $X \subseteq \bigcup_{\alpha \in \Lambda}A_{\alpha}$, there are $A_{\alpha_1},\cdots, A_{\alpha_n} \in \{A_\alpha \}_{\alpha \in \Lambda}$
such that $X \subseteq  \bigcup_{i=1}^{n}A_{\alpha_{i}}$. A topological space $X$ is said to be a \textit{Hausdorff space (T2-space)} if each pair $x, y$ of distinct points of $X$ there exist	neighborhoods $U_1$ and $U_2$ of $x$ and $y$ respectively, that are disjoint.

\begin{thm}\label{1.13} Let $Q$ be a poset in $\mathbb{P}^{\ell}_{MFP}$ such that $Q$ satisfies $a.c.$ If $Min(Q)$ is compact, then for any $x\in Q$, there exists $y\in Q$ such that $\overline{int V(x)} = \overline{D(y)}$.
\end{thm}
\begin{proof} Let $x\in Q$. Then it is enough to prove that there exists $y\in x^{\perp}$ such that $V'(x)\cap V'(y)=\emptyset $ and hence by using the proof of Theorem \ref{1.5}, we have $\overline{int V(x)} = \overline{D(y)}$. Clearly, $V'(x)\cap V'(x^{\perp})=\emptyset$. This gives that $V'(x)\cap (\bigcap _{a\in x^{\perp}}V'(a))= \emptyset $. Since $Min(Q)$ is compact, it satisfies the finite intersection property. Thus there exist $a_{1}, a_{2},\cdots ,a_{n}\in x^{\perp}$ such that $V'(x)\cap (\bigcap _{i=1}^{n}V'(a_i))= \emptyset $, where $a_i \in x^{\perp}$. Then by the Theorem \ref{2.3'}, we can easily prove that $V'(x^{\perp})=\bigcap _{i=1}^{n}V'(a_i), $  where $a_i \in x^{\perp}$. Now, as we have $Q$ satisfies $a.c.$, by the induction, we can prove that there exists $y\in Q$ such that  $y^{\perp}= \bigcap _{i=1}^{n} a_i^{\perp}$. Clearly, $x\in  \bigcap _{i=1}^{n} a_i^{\perp} = y^{\perp}$, we have $y\in x^{\perp}$. We prove that $V'(y)=\bigcap _{i=1}^{n}V'(a_i)=V'(x^{\perp})$. For this, let $M\in V'(y)$. This gives $y^{\perp}\not \subseteq M $ and hence $\bigcap _{i=1}^{n} a_i^{\perp} \not \subseteq M$. Thus  $a_{i}\in M$, for every $i=1,2,\cdots, n$ and hence $M\in \bigcap _{i=1}^{n}V'(a_i)=V'(x^{\perp})$. This proves that $V'(y)\subseteq V'(x^{\perp})$. Now, let $M\in \bigcap _{i=1}^{n}V'(a_i)=V'(x^{\perp})$ and since $y\in x^{\perp}$, we have $M\in V'(y)$. Thus we have $V'(y)=\bigcap _{i=1}^{n}V'(a_i)=V'(x^{\perp})$ and hence $V'(x)\cap V'(y)=\emptyset$. 
\end{proof}

\begin{lemma}\label{1.16} Let $Q$ be a poset in $\mathbb{P}^{\ell}_{MFP}$. Then for any elements $x_1,x_2,\cdots ,x_n$ in $Q$, \linebreak $(\{x_1,x_2,\cdots , x_n\}^{u\ell})^{\perp }=x_1^{\perp }\cap x_2^{\perp }\cap \cdots \cap x_n^{\perp }$.\end{lemma}

\begin{proof} Let $Q\in \mathbb{P}^{\ell}_{MFP}$ and let $\{x_1,x_2,\ldots, x_n\} \subseteq Q$. Since $x_i\in \{x_1,x_2,\ldots, x_n\}^{u\ell}$ for all $i=1,2,\ldots, n$, we have $(\{x_1,x_2,\ldots, x_n\}^{u\ell})^{\perp } \subseteq x_i^{\perp }$ for all $i=1,2,\ldots, n$ and hence $(\{x_1,x_2,\ldots, x_n\}^{u\ell})^{\perp } \subseteq x_1^{\perp }\cap x_2^{\perp }\cap \cdots \cap x_n^{\perp }$. Now, we prove the reverse inclusion $ x_1^{\perp }\cap x_2^{\perp }\cap \cdots \cap x_n^{\perp } \subseteq (\{x_1,x_2,\ldots, x_n\}^{u\ell})^{\perp } $. Let $t\in x_1^{\perp }\cap x_2^{\perp }\cap \cdots \cap x_n^{\perp }$. We claim that $\{t,z\}^{\ell}=\{0\},~ \text{for all}~ z\in \{x_1,x_2,\ldots, x_n\}^{u\ell}$. Suppose on the contrary that there exists an element $q\in \{x_1,x_2,\ldots, x_n\}^{u\ell}$ such that $\{t,q\}^{\ell}\neq \{0\}$. Thus there is a non-zero element $s\in \{t,q\}^{\ell}$. Clearly, $s\in \{x_1,x_2,\ldots, x_n\}^{u\ell}$. Since $s\neq 0$, by Zorn’s Lemma, there is a maximal $\ell$-filter $F$ containing $s$. By the assumption that $Q\in \mathbb{P}^{\ell}_{MFP}$, $F$ is a prime filter. Clearly, $t\in F$. This together with $t\in x_1^{\perp }\cap x_2^{\perp }\cap \cdots \cap x_n^{\perp }$ and $F$ is an $\ell$-filter, we have $x_i\not \in F$, for each $i=1,2,\ldots , n$. Now, we show that there exists an element $z\in \{x_1,x_2,\ldots ,x_n\}^{u}$ such that $z\not \in F$ by using primeness of the filter $F$ for $(n-1)$ times. For simplicity, we can take the case when $n=3$. Since $x_1,x_2,x_3 \not \in F$ and $F$ is prime, we have $\{x_1,x_2\}^{u}\not \subseteq F$ and $\{x_1,x_3\}^{u}\not \subseteq F$. Thus, there exist elements $y_1\in \{x_1,x_2\}^{u}$ and $y_2\in \{x_1,x_3\}^{u}$ such that $y_1, y_2 \not \in F$. Again by primeness of the filter $F$, we have $\{y_1,y_2\}^{u}\not \subseteq F$. Thus there exists an element $z_1\in \{y_1,y_2\}^{u}$ such that $z_1\not \in F$. Clearly, $z_1\in \{x_1,x_2,x_3\}^{u} \setminus F $.   Thus, repeating this technique,  we can find an element $z\in \{x_1,x_2,\ldots,x_n\}^{u}$ such that $z\not \in F$.  As  $s\in \{x_1,x_2,\ldots , x_n\}^{u\ell}$ and $z\in \{x_1,x_2,\ldots ,x_n\}^{u},$ we have  $z\leq s.$ Then $z\not \in F$ implies that $s\not \in F$, a contradiction to the fact that $s\in F$. Thus, we have $s=0.$ Therefore, $\{t,q\}^{\ell}=\{0\},$ for every $q\in \{x_1,x_2,\cdots, x_n\}^{u\ell}$ and hence $t\in (\{x_1,x_2,\ldots , x_n\}^{u\ell})^{\perp }$. Thus, we have $(\{x_1,x_2,\ldots , x_n\}^{u\ell})^{\perp }=x_1^{\perp }\cap x_2^{\perp }\cap \cdots \cap x_n^{\perp }$.
\end{proof}

\begin{rem}
	We show that the assumptions $Q\in \mathbb{P}^{\ell}_{MFP}$ of Lemma \ref{1.16} is necessary. For the poset $Q$ depicted in Figure 1, we have $(\{a, b\}^{u\ell})^{\perp } \neq a^{\perp }\cap b^{\perp }$ and $Q\notin \mathbb{P}^{\ell}_{MFP}$. In fact, a maximal filter $F=[b)$ is not prime. 
\end{rem}	

\begin{center}
	\begin{tikzpicture}[scale=.91]
		\draw [fill=black] (-2.5,0) circle (.05);
		\draw [fill=black] (-2.5,0.5) circle (.05);
		\draw [fill=black] (-2.5,1) circle (.05);
		\draw [fill=black] (-2.5,1.5) circle (.05);
		\draw [fill=black] (-1.3,0.5) circle (.05);
		\draw [fill=black] (-1.3,1) circle (.05);
		\draw [fill=black] (-1.3,1.5) circle (.05);
		\draw [fill=black] (-1.3,2) circle (.05);
		\draw [fill=black] (-1.3,-1.6) circle (.05);
		\draw [fill=black] (-0.5,-1.6) circle (.05);
		\draw [fill=black] (-1.3,-2.3) circle (.05);
		
		\draw [line width=1pt] (-2.5,0)--(-2.5,0.5)--(-2.5,1)--(-2.5,1.5);
		\draw [line width=1pt] (-1.3,0.5)--(-1.3,1)--(-1.3,1.5)--(-1.3,2);
		\draw [line width=1pt] (-2.5,0)--(-1.3,0.5);
		\draw [line width=1pt] (-2.5,0.5)--(-1.3,1);
		\draw [line width=1pt] (-2.5,1)--(-1.3,1.5);
		\draw [line width=1pt] (-2.5,1.5)--(-1.3,2);
		
		\draw [line width=1pt] (-1.3,-1.6)--(-1.3,-2.3);
		\draw [line width=1pt] (-0.5,-1.6)--(-1.3,-2.3);
		\draw [line width=1pt] (-1.3,-0.9)--(-1.3,-1.5);
		\draw [line width=1pt] (-1.2,-0.9)--(-0.5,-1.6);
		\draw [line width=1pt] (-2.5,-1.2)--(-1.3,-2.3);
		
		\draw[dashed, line width=1pt, -{Triangle[width=3pt, length=3.5pt]}]  (-1.3,0.5) to [out=-90, looseness=0]  (-1.3,-0.7);
		\draw[dashed, line width=1pt, -{Triangle[width=3pt, length=3.5pt]}]  (-2.5,0) to [out=-90, looseness=0]  (-2.5,-1);
		
		\node at (-1.3,-2.5) {$0$};
		\node at (-1.1,-1.6) {$b$};
		\node at (-0.3,-1.6) {$c$};
		\node at (-2.7,1.5) {$a$};
	\end{tikzpicture}
	\vspace{.01in}
	
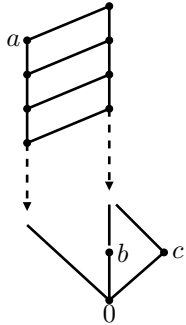
\captionof{figure}{A poset $Q\notin \mathbb{P}^{\ell}_{MFP}$ such that $(\{a, b\}^{u\ell})^{\perp } \neq a^{\perp }\cap b^{\perp }$.}
	\label{figure1}
\end{center}

Clearly, every quasi-complemented poset is 
weakly quasi-complemented. The converse is true for $0$-distributive join-semilattice; see Mundlik et al. \cite[Lemma 3.34]{MJH}. Now, we extend this result to posets.

\begin{thm}\label{1.20}
	Let $Q$ be a poset in $\mathbb{P}^{\ell}_{MFP}$ such that $Q$ satisfies $a.c.$ Then $Q$ is a quasi-complemented poset if and only if $Q$ is a weakly quasi-complemented poset. 
\end{thm}

\begin{proof} The only if implication follows by Theorem \ref{2.3'} and by the definition of a weakly quasi-complemented poset.
	For the converse part, assume that $Q$ is a weakly quasi-complemented poset. Let $x\in Q$. Thus there exist $y_1, y_2,\cdots, y_n$ distinct from $x$ such that $x^{\perp \perp} = \bigcap_{i=1}^{n}y_{i}^{\perp }$. As $y_1, y_2\in Q$ and $Q$ satisfies the $a.c.$,  there exists $z_1\in Q$ such that $y_1^{\perp }\cap y_2^{\perp }=z_1^{\perp }$. Also, $z_1, y_3\in Q$, and $Q$ satisfies the $a.c.$, there exists $z_2\in Q$ such that $z_1^{\perp }\cap y_3^{\perp }=z_2^{\perp }$. Continuing in this way, we can find an element $z_{n-1}\in Q$ such that  $z_{n-2}^{\perp }\cap y_n^{\perp }=z_{n-1}^{\perp }$. Clearly, $z_{n-1}^{\perp } = y_1^{\perp }\cap y_2^{\perp }\cap \cdots \cap y_n^{\perp } =x^{\perp \perp} $. Thus $z_{n-1}^{\perp \perp }=x^{\perp \perp \perp}=x^{\perp }$. By Theorem \ref{2.3'}, $Q$ is a quasi-complemented poset. 
\end{proof}

The following result is due to Mundlik et al. \cite {MJH}, which gives the necessary condition for compactness of $Min(Q)$, which extends the result of Pawar \cite{PT 2}.

\begin{thm}[Mundlik et al. \cite {MJH}, Theorem 3.35]\label{3.35} Let $Q$ be a poset in $\mathbb{P}^{\ell}_{MFP}$. If $Min(Q)$ is compact, then $Q$ is a weakly quasi-complemented.   
\end{thm}

\begin{rem}[Mundlik et al. \cite{MJH}, Remark 3.3]\label{2.27} It is clear that if $Q\in \mathbb{P}^{\ell}_{MFP}$, then the set of all minimal prime semi-ideals $Min_s(Q)$ and set of all minimal prime ideals $Min(Q)$ coincides. That is $Min(Q) = Min_s(Q)$ for $Q\in \mathbb{P}^{\ell}_{MFP}$.
\end{rem}

\begin{lem}[Cornish \cite{C}, Proposition 1.2]\label{2.20}
	A $0$-distributive lattice $L$ is a quasi-complemented if and only if $Min(L)$ is a compact Hausdorff space.
\end{lem}

We close this section with the following Main Theorem. 
\vspace{0.4mm}

\begin{theorem*} \label{1}
	Consider the following statements for a poset $Q\in \mathbb{P}^{\ell}_{MFP}$.
	\begin{enumerate} 
		\item[(1)] $Q$ is a quasi-complemented poset.
		\item[(2)] For any $x \in Q$, there exists $y \in Q$ such that $x^{\perp}= y^{\perp \perp}$.
		\item[(3)] For every $x\in Q$, there exists $y\in x^{\perp}$ such that $x^{\perp}\cap y^{\perp}=\{0\}$.
		\item[(4)] For every $x\in Q$, there exists $y\in Q$ such that $\overline{\text{int}(V(x))}=\overline{D(y)}$.
		\item[(5)] For every $x\in Q$, there exists $y\in Q$ such that $V'(y)=V'(x^{\perp})$.
		\item[(6)] $\Gamma(Q)$ is a complemented.
		\item[(7)] $\Gamma(Q)$ is an uniquely complemented.
		\item[(8)] $Q$ is a weakly quasi-complemented.
		\item[(9)] $Min(Q)  ( = Min_s(Q))$ is compact.
	\end{enumerate}
	Then the implications $(1)\Leftrightarrow (2) \Leftrightarrow (3) \Leftrightarrow (4) \Leftrightarrow (5)$, $(3)\Rightarrow (6)\Leftrightarrow (7)$, $(1)\Rightarrow (8)$, and $(9)\Rightarrow (8)$ hold. Moreover, if the poset $Q$ satisfies $a.c.$, then the statements $(1)$ to $(8)$ are  equivalent. Further,  if $Q$ is a lattice, then statements $(1)$ to $(9)$ are  equivalent.
\end{theorem*}

\begin{proof} Let $Q$ be a poset in $\mathbb{P}^{\ell}_{MFP}$.
Clearly, by Theorem \ref{2.3'} and Theorem \ref{1.5}, the first four statements are all equivalent. Furthermore, statements $(3)$ and $(5)$ are equivalent by Corollary \ref{1.8}. Thus, statements $(1)$ through $(5)$ are all equivalent.
	
Now, the implication $(3) \Rightarrow (6)$ follows by the assumption and from the definition of the complemented zero-divisor graph of $Q$. Moreover, by Theorem \ref{1.11}, statements $(6)$ and $(7)$ are equivalent.
	
The implication $(1) \Rightarrow (8)$ follows directly from Theorem \ref{2.3'} and the definition of a weakly quasi-complemented poset. Also, the implication $(9) \Rightarrow (8)$ follows from Theorem \ref{3.35} and Remark \ref{2.27}.
	
Now, assume that the poset $Q$ satisfies the condition $a.c.$ Then, by Theorem \ref{1.11'}, statements $(1)$ and $(6)$ are equivalent. Additionally, the converse of $(1) \Rightarrow (8)$ follows from Theorem \ref{1.20}. Therefore, under the assumption that $Q \in \mathbb{P}^{\ell}_{MFP}$ satisfies $a.c.$, all the statements $(1)$ to $(8)$ are equivalent.
	
Furthermore, if $Q$ is a lattice, then by Theorem \ref{2.3}, Lemma \ref{2.20}, and Corollary \ref{1.12'}, all the statements $(1)$ to $(9)$ are equivalent.
\end{proof}

\begin{rem} Note that, if we drop the condition that $Q$ satisfies $a.c.$ (or $Q$ is a lattice) then the reverse implications $(6)\Rightarrow (3), (8)\Rightarrow (1)$ need not be true. Consider a meet-semilattice $S=\{0,a,b\}$ with $0\leq a,b$ and $a$ and $b$ are incomparable. Then $(S,\leq )$ is $0$-distributive meet-semilattice and hence by Theorem \ref{2.3}, $S\in \mathbb{P}^{\ell}_{MFP}$. Clearly, $\Gamma(S)$ is complemented but for every $y\in S=0^{\perp}$, $0^{\perp}\cap y^{\perp}\neq \{0\}$. Also, $S$ does not satisfy $a.c.$   Thus $(6)\nRightarrow (3)$. Also, it is easy to see that $S$ is a weakly quasi-complemented poset but not  quasi-complemented. Thus $(8)\nRightarrow (1)$.  Also, the reverse implications $(8)\Rightarrow (9)$ need not be true if $Q$ does not satisfy $a.c.$
	
\begin{center}
	\begin{tikzpicture}[scale=1]
		\draw[line width=1pt] (-0.01,0.01) -- (-2,1);
		\draw[line width=1pt] (0,0) -- (-1,1);
		\draw[line width=1pt] (0,0) -- (1,1);
		\draw[line width=1pt] (0,0) -- (2,1);
		
		\draw[fill=black] (0,0) circle (1.6pt) node[anchor=north] {0};
		\draw[fill=black] (-2,1) circle (1.6pt) node[anchor=east] {$x_1$};
		\draw[fill=black] (-1,1) circle (1.6pt) node[anchor=west] {$x_2$};
		\draw[fill=black] (1,1) circle (1.6pt) node[anchor=west] {$x_3$};
		\draw[fill=black] (2,1) circle (1.6pt) node[anchor=west] {$x_4$};
		
		\draw[line width=1pt] (0,4) -- (-2,3);
		\draw[line width=1pt] (0,4) -- (-1,3);
		\draw[line width=1pt] (0,4) -- (1,3);
		\draw[line width=1pt] (0,4) -- (2,3);
		
		\draw[fill=black] (0,4) circle (1.6pt) node[anchor=south] {1};
		\draw[fill=black] (-2,3) circle (1.6pt) node[anchor=east] {$x'_1$};
		\draw[fill=black] (-1,3) circle (1.6pt) node[anchor=west] {$x'_2$};
		\draw[fill=black] (1,3) circle (1.6pt) node[anchor=west] {$x'_3$};
		\draw[fill=black] (2,3) circle (1.6pt) node[anchor=west] {$x'_4$};
		
		\draw[line width=1pt] (-2,1) -- (-1,3);
		\draw[line width=1pt] (-2,1) -- (1,3);
		\draw[line width=1pt] (-2,1) -- (2,3);
		\draw[line width=1pt] (-1,1) -- (-2,3);
		\draw[line width=1pt] (-1,1) -- (1,3);
		\draw[line width=1pt] (-1,1) -- (2,3);
		\draw[line width=1pt] (1,1) -- (-2,3);
		\draw[line width=1pt] (1,1) -- (-1,3);
		\draw[line width=1pt] (1,1) -- (2,3);
		\draw[line width=1pt] (2,1) -- (-2,3);
		\draw[line width=1pt] (2,1) -- (-1,3);
		\draw[line width=1pt] (2,1) -- (1,3);
		
		\draw[
		line width=1pt,
		black,
		line cap=round,
		dash pattern=on 0pt off 2pt
		] (2.6, 1) -- (3.5, 1);
		\draw[
		line width=1pt,
		black,
		line cap=round,
		dash pattern=on 0pt off 2pt
		] (2.6, 3) -- (3.5, 3);
	\end{tikzpicture}
\end{center}
\vspace{-.25in}

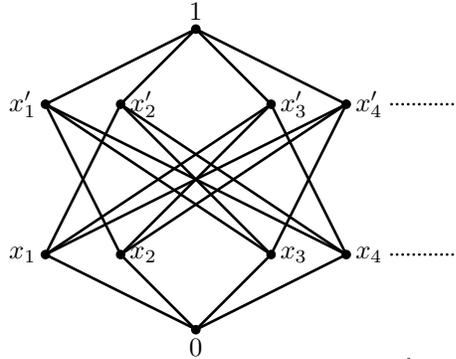
\captionof{figure}{A quasicomplemented poset $Q$ in $\mathbb{P}^{\ell}_{MFP}$ without $a.c.$}
\label{figure2}
	
	Let $Q$ be a poset under set inclusion, where $Q=\{X\subseteq \mathbb{N}| \text{ either } X=\{n\} \text{ or } X=\mathbb{N}\setminus\{n\} \text{ for every } n \in \mathbb{N}\}$. The Hasse diagram of $Q$ is shown in Figure 2. It is easy to observe that, $Q$ is quasicomplemented (and hence, weakly quasicomplemented) without $a.c.$, and $Min(Q)$ is not compact.
\end{rem}
Thus, we raise the following conjecture,
\begin{cthm*} Let $Q$ be a quasicomplemented poset in $\mathbb{P}^{\ell}_{MFP}$ with  $a.c.$ Is $Min(Q)$ compact?
\end{cthm*}

\section{Applications}

In this section, we show that the zero-divisor graph of a poset provides a powerful tool for visualizing and analyzing the structure of graphs associated with rings, in particular comaximal (ideal) graphs and zero-divisor graphs. Using Main Theorem, we prove that the comaximal (ideal) graphs and zero-divisor graphs of Artinian rings are always complemented (respectively, uniquely complemented).

Let $R$ be a commutative ring with identity, and $Id(R)$ denotes the set of all ideals of a ring $R$. As we know, $Id(R)$ forms a complete lattice with respect to the inclusion as the partial order and $(0)$ and $R$ as its least and the greatest element, respectively.

\begin{defn}[Ye and Wu \cite{YW}]
	Let $R$ be a commutative ring with identity. We associate a simple undirected graph with $R$, called as the \textit{comaximal ideal graph} $\mathbb{CIG}(R)$ of $R$, where the vertices of $\mathbb{CIG}(R)$
	are proper ideals of $R$ that are not contained in the Jacobson radical $J(R)$ of $R$ and two vertices $I$ and	$J$ are adjacent if and only if $I + J = R$.
\end{defn}

In \cite{KJ}, Khandekar and Joshi proved that the comaximal ideal graph of commutative ring $R$ is nothing but the zero-divisor graph of the dual of the lattice $L=Id(R)$, where $Id(R)$ is the set of all ideals of $R$. Note that the dual of lattice $L$, denoted by $L^\partial$, is a $0$-distributive lattice.

\begin{thm}[Khandekar et al. \cite{KJ}, Theorem 5.1 \label{kj}] Let $R$ be a commutative ring with identity and let $L^\partial$ be the dual of the lattice $Id(R)$ of all ideals of $R$. Then $\mathbb{CIG}(R) = \Gamma(L^\partial).$
\end{thm}

The following discussion can be found in Khandekar et al. \cite{KJ}. 
Let $R$ be an Artinian ring. Then $R =\displaystyle \prod_{i=1}^{n}{R_i}$, where $R_i$
is an Artinian local ring for every $i$. It is easy to observe that $R$ has identity if and only if $R_i$ has identity for all $i$. Hence $Id(R) = Id(R_1) \times Id(R_2) \times \cdots \times Id(R_n)$, where
$Id(R_1) \times Id(R_2) \times \cdots \times Id(R_n)$ denotes the (lattice) direct product of the ideal lattices $Id(R_i)$ of $R_i$.
Hence, $Id(R)^\partial = Id(R_1)^\partial \times Id(R_2)^\partial \times \cdots \times  Id(R_n)^\partial$. Clearly from Theorem \ref{2.3}, Theorem \ref{2.4'} and Lemma \ref{3.8}, $Id(R_i)$ for each $i$, is a dually pseudocomplemented lattice and hence $\displaystyle \prod_{i=1}^{n}{Id(R_i)}$ is also a dually pseudocomplemented lattice, that is, the dual concept of a pseudocomplemented lattice. Thus, $Id(R)^\partial$ is a pseudocomplemented lattice and hence a quasicomplemeted lattice. By Main theorem, $\Gamma(Id(R)^\partial)$ is a complemented and therefore by Theorem \ref{kj} comaximal ideal graph $\mathbb{CIG}(R)$ is always a complemented, equivalently a uniquely complemented graph. 

\begin{thm} Let $R$ be an Artinian ring. The comaximal ideal graph $\mathbb{CIG}(R)$ is a complemented graph, (equivalently a uniquely complemented graph).  
\end{thm}

In 1995, Sharma and Bhatwadekar \cite{SB} introduced a graph
$\mathbb{CG}(R)$ on a commutative ring $R$ with identity, whose vertices are
the elements of $R$ and two distinct vertices $x$ and $y$ are adjacent
if and only if $Rx + Ry = R$.
Maimani et al. \cite{MSSY}	named this graph $\mathbb{CG}(R)$ as the \textit{comaximal graph} of $R$.
\begin{thm}[Ghadge et al.\cite{GKJ}, Theorem 3.16 \label{pg}]
	Let $R$ be a finite commutative ring with identity.	Then $\mathbb{CG}(R)=\Gamma(\mathbb{L})$.  
\end{thm}

In \cite{GKJ}, it is proved that the comaximal graph of an Artinian ring is the zero-divisor graph of a specially constructed lattice $\mathbb{L}$. The construction of $\mathbb{L}$ is given in \cite[Construction 3.6]{GKJ}. By \cite[Corollary 3.9 and Theorem 3.16]{GKJ}, it is clear that $\mathbb{L}$
is a pseudocomplemented lattice and hence it is a quasicomplemeted lattice. By  Main theorem, $\Gamma(\mathbb{L})$ is complemented, and therefore by Theorem \ref{pg} the comaximal graph $\mathbb{CG}(R)$ is always  complemented, equivalently a uniquely complemented graph. 

\begin{thm} Let $R$ be an Artinian ring. The comaximal graph $\mathbb{CG}(R)$ is a complemented, (equivalently a uniquely complemented graph).  
\end{thm}

 Angsuman Das \cite{AD} defined and studied the \textit{nonzero component union graph} (also known as \textit{skeleton graph}) $\mathbb{UG(V)}$ of a finite dimensional vector space $\mathbb{V}$ over the field $\mathbb{F}$ with respect to the basis $\mathcal{B} = \{v_1, \cdots , v_n\}$ as follows: The vertex set of graph $\mathbb{UG(V)}$ is $\mathbb{V} \setminus \{0\}$ and for any $a, b \in \mathbb{V} \setminus \{0\}$, $a$ is adjacent to $b$ if and only if $S_\mathcal{B}(a) \cup S_\mathcal{B}(b) = \mathcal{B}$, where $S_\mathcal{B}(a)=\{v_i ~|~ a_i \neq 0, ~a = a_1v_1 +\cdots + a_nv_n\}$.
 
  Recently, Khandekar et al. \cite{KCJ} proved that the nonzero component union graph of a finite dimensional vector spaces can be studied via the zero-divisor graphs of posets. For this result, we need the concept of join of graphs.
	
	The \textit{join of two graphs} $G$ and $H$ is the graph formed from disjoint copies of $G$ and $H$ by connecting each vertex of $G$ to each vertex of $H$. We denote the join of graphs $G$ and $H$
	by $G \vee H$; see West \cite{W}.

\begin{thm}[Khandekar et al. \cite{KCJ}, Theorem 3.4] $\mathbb{UG(V)} = \Gamma(\mathcal{L^\partial}) \vee K_t$, where $t = |\mathbb{V}_{12 \cdots n}| = (|F|-1)^n$ and $K_t$ is a complete graph on $t$ vertices.
\end{thm}

In the above Theorem, a lattice $\mathcal{L}^\partial$ is constructed which is pseudocomplemented and hence quasicomplemented. Therefore by Main Theorem, $\Gamma(\mathcal{L^\partial})$ is a complemented graph. However, since $\mathbb{UG(V)}$ is the join of $ \Gamma(\mathcal{L^\partial})$ and the complete graph  $K_t$,  it is easy to observe that  $\Gamma(\mathcal{L^\partial}) \vee K_t$ is not a complemented graph. Therefore, $\mathbb{UG(V)}$ is not  complemented, and hence not a uniquely complemented graph. 

\begin{thm}
	The nonzero component union graph $\mathbb{UG(V)}$ of a finite dimensional vector space $\mathbb{V}$ over a field $\mathbb{F}$ with respect to the basis $\mathcal{B}$ is neither  complemented nor  uniquely complemented. 
\end{thm}

Now, we apply Main Theorem to the zero-divisor graphs of a reduced (multiplicative) semigroup $S$ with $0$, meaning that $S$ contains no non-zero nilpotent elements. Let $S$ be a semigroup with $0$ and $1$. We associate to the semigroup $S$ a simple undirected graph $G(S)$, the zero-divisor graph of $S$, with the vertex set	 $V(G(S)) = \{x\in S\setminus\{0\} ~|~ xy=0$ for some $y\in S\setminus\{0\}\}$
and two vertices $a$ and $b$ in $V(G(S))$ are adjacent if and only if $ab=0$.

Let $S$ be a reduced (multiplicative) commutative semigroup with $0\neq 1$. Define
a relation $\leq $ such that $r\leq s$ in $S$ if and only if either $ann(s) \subsetneqq ann(r) $ or $r \leq s$ in some
predetermined well-order on the set $[r] = \{x \in S ~|~ ann(r) = ann(x)\}$. Note that $ann(s) =\{x \in S\; |\; xs=0 \}$. LaGrange and Roy \cite[Remark 3.4]{LR} proved that $\leq $ is a partial order on $S$.
We recall the following result due to  Kadu et al. \cite{KJG}.

\begin{theorem}[Kadu et al. \cite{KJG}, Theorem 3.4]\label{2.1} Let $S$ be a reduced (multiplicative) commutative semigroup. Then the following statements are true.
	\begin{enumerate}\item	$(S; \leq ) $ is a meet-semilattice. \item  For $a, b \in S$, we have $ab = 0$ if and only if $a\wedge b = 0$. In this case, $ann(a)=a^{\perp}$. Therefore, the zero-divisor graph $G(S)$ of a semigroup $S$ and the zero-divisor graph $\Gamma (S)$ of $S$ (treated as a meet-semilattice) are essentially the same, that is, $\Gamma (S) = G(S)$. 
	\end{enumerate}
\end{theorem}

In view of the above result, it is evident that a reduced (multiplicative) commutative semigroup with $0$ satisfies the annihilator condition ($a.c.$) if and only if the associated poset derived from $S$ also satisfies $a.c.$ The following result is due to Devhare et al. \cite{DJL}.

\begin{theorem}[Devhare et al. \cite{DJL}, Theorem 3.4]\label{2.2} Let $S$ be a reduced (multiplicative)  commutative semigroup with $0$. If $S$ satisfies $a.c.$, then the partially ordered set $S$ is $0$-distributive.
\end{theorem}

Hence, the following result is immediate from Theorem \ref{2.1},  Theorem \ref{2.2}, and Theorem \ref{2.3}.

\begin{theorem}\label{2.4} Let $S$ be a (multiplicative) reduced commutative semigroup with $0$. If $S$ satisfies $a.c$., then $S\in \mathbb{P}^{\ell}_{MFP}$ under the partial order mentioned before Theorem \ref{2.1} and $\Gamma (S) = G(S)$. 
\end{theorem}

In view of Theorem \ref{2.4} and Main Theorem, we have the following result.

\begin{theorem} Let $S$ be a (multiplicative) reduced commutative semigroup with $0$. If $S$ satisfies $a.c.$, then following statements are equivalent;
	\begin{enumerate}
		\item[(1)] $S$ is a quasi-complemented meet-semilattice under the partial order given before Theorem \ref{2.1}.
		\item[(2)] 	for any $x \in S$, there exists $y \in S$ such that $x^{\perp}= y^{\perp \perp}$.
		\item[(3)] for every $x\in S$, there exists $y\in x^{\perp}$ such that $x^{\perp}\cap y^{\perp}=\{0\}$.
		\item[(4)] for every $x\in S$, there exists $y\in x^{\perp}$ such that $\overline{int(V(x))}=\overline{D(y)}$.
		\item[(5)] for every $x\in S$, there exists $y\in S$ such that $V'(y)=V'(x^{\perp})$.
		\item[(6)] $\Gamma(S)$ is complemented.
		\item[(7)] $\Gamma(S)$ is an uniquely complemented.
		\item[(8)] $G(S)$ is complemented.
		\item[(9)] $G(S)$ is an uniquely complemented.
		\item[(10)] $S$ is a weakly quasi-complemented meet-semilattice.
	\end{enumerate}	
	
\end{theorem}

\noindent{\bf Acknowledgments:}
The authors thanks both the referees for their suggestions which improved the presentation of the paper.
	
	The third author is financially supported by the Science and Engineering Research Board (DST) via Project CRG/2022/002184.


\begin{thebibliography}{99}
	\bibitem{AN} D. D.	Anderson and M. Naseer, {\it Beck’s coloring of a commutative ring},	J. Algebra, {\bf 159(2)} (1993), 500–514.
	
	\bibitem{AL} D. F. Anderson and P. S. Livingstone, {\it The zero-divisor graph of a commutative ring}, J. Algebra, {\bf 217(2)} (1999), 434-447.
	
	\bibitem{ALS} D. F. Anderson, R. Levy, J.Shapiro, {\it Zero-divisor graphs, von Newmann regular rings, and Boolean Algebras}, J. Pure Appl. Algebra, {\bf 180(3)} (2003), 221-241.
	
	
	\bibitem{B} I. Beck,  {\it Coloring of commutative rings}, J. Algebra, {\bf 116(1)} (1988), 208-226.
	
	\bibitem{CPRL} C. Bender, P. Cappaert, R. DeCoste, and L. DeMeyer, {\it Complemented zero-divisor graphs associated with finite commutative semigroups}, Comm. Algebra, \textbf{52(7)} (2024), 2852-2867.
	
	\bibitem{BAL} D. Bennis, B. E. Alaoui, and R. L'Hamri,  {\it Rings whose associated extended zero divisor graphs are complemented}, Bull. Korean Math. Soc., {\bf 61(3)} (2024), 763-777.
	
	
	\bibitem{C} W. H. Cornish, {\it quasi-complemented lattices}, Comment. Math. Univ. Carolin. \textbf{15(3)} (1974), 501-511.
	
	\bibitem{AD} A. Das, \textit{Nonzero Component Union-Graph of a Finite Dimensional	Vector Space.}, Linear Multilinear Algebra, \textbf{65(6)}
	(2017), 1276-1287.
	
	\bibitem{DMS} F. DeMeyer, T. McKenzie and K. Schneider, {\it The zero-divisor graph of a commutative semigroup}, Semigroup Forum, {\bf 65(2)} (2002), 206-214.
	
	\bibitem{DJL} S. U. Devhare, Vinayak Joshi and J. D. LaGrange,  {\it On the connectedness of the complement of the zero-divisor graph of a poset}, Quaest. Math., {\bf 42(7)} (2019), 939-951. 
	
	\bibitem{GKJ} P. Gadge, N. Khandekar and Vinayak Joshi, {\it On the comaximal graph of a ring}, AKCE Int. J. Graphs Comb., {\bf 21(2)} (2024), 143-151.
	
	\bibitem{G} G. Gr\"atzer, {\it General Lattice Theory}, Birkh\"auser Basel (1998).
	
	\bibitem{GV} P. A. Grillet and J. C. Varlet, {\it Complementedness conditions in lattices}, Bull. Soc. Roy. Sci. Liège, {\bf 36} (1967), 628-642.
		
	\bibitem{H}	R. Hala\v{s},  {\it Pseudocomplemented ordered sets}, Arch. Math. (Brno), {\bf 29(3-4)} (1993), 153–160.
	
	\bibitem{HJ} R. Hala\v{s} and M. Jukl, {\it On Beck’s coloring of posets}, Discrete Math., {\bf 309(13)} (2009), 4584–4589.
	
		
	\bibitem{Jr}	C. Jayaram, {\it Quasicomplemented semilattices}, Acta Math. Acad. Sci. Hung., {\bf 39} (1982), 39-47.
	
		
	\bibitem{J} Vinayak Joshi, {\it Zero-divisor graph of a poset with respect to an ideal}, Order, {\bf 29(3)} (2012), 499-506.
	
	\bibitem{VK} Vinayak Joshi and M. Khalid, {\it Quassi-complemented posets}, Asian-Eur. J. Math., {\bf 15(10)} (2022), 2250183.
	
	\bibitem{JA 1} Vinayak Joshi  and A. U. Khiste, {\it On the zero-divisor graph of a pm-lattice}, Discrete Math., {\bf 312(12-13)} (2012), 2076-2082.
	
	\bibitem{JA 2} Vinayak Joshi  and A. U. Khiste, {\it  The zero-divisor graphs of Boolean posets}, Math. Slovaca, {\bf 64(2)} (2014), 511–519.
	
	
	\bibitem{VM} Vinayak Joshi and N. D. Mundlik, {\it Baer ideals in $0$-distributive posets}, Asian-Eur. J. Math., {\bf 9(3)} (2016), 1650055 (16 pages).
	
	\bibitem{JW} Vinayak Joshi and B. N. Waphare, {\it Characterizations of $0$-distributive posets}, Math. Bohem., {\bf 130(1)} (2005), 73–80. 
	
	\bibitem{KJG}  G. Kadu, Vinayak Joshi and S. Gonde, {\it On weakly perfect annilating-ideal graphs}, Bull. Aust. Math. Soc., {\bf 104(3)} (2021), 362–372.
	
	\bibitem{KCJ} N. Khandekar, P.J. Cameron and Vinayak Joshi, \textit{Component graphs of vector spaces and zero-divisor graphs of ordered sets}, AKCE Int. J. Graphs Comb.,(2025), https://doi.org/10.1080/09728600.2025.2449683, 1-7.
	
	\bibitem{KJ}  N. Khandekar and Vinayak Joshi, {\it Chordal and perfect zero-divisor graphs of posets and applications to graphs associated with algebraic structures}, Math. Slovaca, {\bf 73(5)} (2023), 1099–1118.	
	
	\bibitem{L}	J. D. LaGrange, {\it Complemented zero divisor graphs and Boolean	rings}, J. Algebra, {\bf 315(2)} (2007), 600-611.
	
	\bibitem{LR}  J. D. LaGrange and K. A. Roy, {\it  Poset graphs and the lattice of graph annihilators}, Discrete Math., {\bf 313(10)} (2013), 1053–1062.
	
	\bibitem{LW} D.	Lu and T. Wu,  {\it The zero-divisor graphs of partially ordered sets and an application to semigroups}, Graphs Combin., {\bf 26(6)} (2010), 793–804.
	
	\bibitem{MSSY} H. Maimani, M. Salimi, A. Sattari, S. Yassemi, {\it Comaximal graph of commutative rings}, J. Algebra, \textbf{319(4)}(2019), 1801-1808.
	
	\bibitem{MJH} N. D. Mundlik, Vinayak Joshi and R. Hala\v{s}, {\it The hull–kernel topology on prime ideals in	posets}, Soft Computing, {\bf 21} (2017), 1653–1665.
	
	\bibitem{NWD} S. K. Nimbhorkar, M. P. Wasadikar and L. Demeyer, {\it  Coloring of meet-semilattices}, Ars Combin., {\bf 84} (2007), 97–104.
	
	
	
	\bibitem{PT 1} Y. S. Pawar and N. K. Thakare,  {\it $0$-Distributive semilattices}, Canad. Math. Bull., {\bf 21(4)}(1978), 469–481.
	
	\bibitem{PT 2}	Y. S. Pawar and N. K. Thakare, {\it The space of minimal prime ideals in a $0$-distributive semilattices}, Period. Math. Hungar., {\bf 13(4)} (1982), 309–319.
	
	\bibitem{SB} P. D. Sharma, S. M. Bhatwadekar, {\it A note on graphical representation of rings}, J. Algebra {\bf 176(1)}, (1995), 124-147.
	
	\bibitem{V} P. V. Venkatanarasimhan, {\it Pseudo-complements in posets}, Proc. Amer. Math. Soc., {\bf 28(1)} (1971), 9–17.
	
	\bibitem{SV} S. Visweswaran, {\it When Is the Complement of the Zero-Divisor Graph of a Commutative Ring Complemented?}, ISRN Algebra, {\bf 4} (2012), 13 pages.
	
	\bibitem{W} D. B. West, {\it Introduction to Graph Theory} (2nd Edition), Prentice Hall, India (2005).

   \bibitem{YW} M. Ye, T. Wu, {\it 	Comaximal ideal graphs of commutative rings}, J. Algebra Appl. 11(6) (2012), Art. ID
   1250114.

	

	
\end{thebibliography}
\end{document}